\documentclass[a4paper,11pt]{article}
\usepackage{amsmath}
\usepackage{amsthm}	
\usepackage{amssymb}
\usepackage{texdraw}
\usepackage{color}
\usepackage{dsfont}
\usepackage{geometry}
\usepackage{graphicx}
\newtheorem {theorem}{Theorem }[section]
\newtheorem {proposition}{Proposition }[section]
\newtheorem {corollary}{Corollary }[section]
\newtheorem {lemma}{Lemma }[section]

\newtheorem {remark}{Remark}[section]

\newtheorem {assumption}{Assumption }[section]

\def\R{{\mathbf R}}
\def\M{{\mathbf M}}
\def\cM{{\cal M}}

\def\cN{{\cal N}}
\def\cC{{\cal C}}
\def\cA{{\cal A}}
\def\cF{{\cal F}}
\title{Existence results for optimal control problems with some special non-linear dependence on state and control}
\author{Pablo Pedregal\thanks{Universidad de Castilla-La Mancha, ETSI Industriales, 13071 Ciudad Real, Spain ({\tt pablo.pedregal@uclm.es}). This work was supported by the research projects MTM2007-62945 of the MCyT (Spain) and PCI08-0084-0424 of the JCCM (Castilla-La Mancha).}  \and  Jorge Tiago\thanks{Universidad de Castilla-La Mancha, ETSI Industriales, 13071 Ciudad Real, Spain ({\tt jorge.tiago@uclm.es}). This research was supported by Funda\c{c}\~ao para a Ci\^encia e a Tecnologia (Portugal)), doctoral degree grant SFRH-BD-22107-2005 and by CMAF-Lisbon University through
FEDER and FCT-Plurianual 2007.}}
\date{}
\begin{document}
\maketitle

\textbf{Abstract}. We present a general approach to prove existence of solutions for optimal control problems not based on typical convexity conditions which quite often are very hard, if not impossible, to check. By taking advantage of several relaxations of the problem, we isolate an assumption which guarantees the existence of solutions of the original optimal control problem. To show the validity of this crucial hypothesis through various means and in various contexts is the main body of this contribution. In each such situation, we end up with some existence result. In particular, we would like to stress a general result that takes advantage of the particular structure of both the cost functional and the state equation. One main motivation for our work here comes from a model for guidance and control of ocean vehicles. Some explicit existence results and comparison examples are given.


\section{Introduction}
This paper focuses on the analysis of optimal control problems of the general form 
\begin{equation}\label{original1}
(P_1)\qquad\qquad \qquad\qquad   \hbox{Minimize in }u:\quad \int_0^T[\sum_{i=1}^s c_i(x(t))\phi_i(u(t))]\, dt \qquad  \qquad\qquad  \qquad
\end{equation}

subject to
\begin{equation}
\label{original2}
x'(t)=\sum_{i=1}^s Q_i(x(t))\phi_i(u(t))\textrm{ in }(0,T), \quad x(0)=x_0\in\R^N,
\end{equation}

and
\begin{equation}
\label{original3}
u\in L^\infty(0, T),\quad u(t)\in K,
\end{equation}
where $K\subset\R^m$ is compact. The state 
$x:(0, T)\to\R^N$ takes values in $\R^N$. 

The mappings
$$
c_i:\R^N\to\R,\quad \phi_i:\R^m\to\R,\quad Q_i:\R^N\to\R^N
$$
as well as the restriction set $K\subset\R^m$ will play a fundamental role. We assume, at this initial stage, that $c_i$ are continuous, $\phi_i$ are of class $\cC^1$, and each $Q_i$ is Lipschitz so that the state system is well-posed. 

In such a general form, we cannot apply results for non-necessarily convex problems like the ones in \cite{B}, \cite{C}, \cite{Ry} or \cite{RbkS}. Besides, techniques based on Bauer' Maximum Principle (\cite{Br}) are quite difficult to extend to our general setting because it is hard to analyze the concavity of the cost functional when the dependence on both state and control comes in product form. Also the Rockafellar's variational reformulation introduced in  \cite{R1}, and well-described in \cite{C1}, \cite{ET} or recently in \cite{P} or \cite{PT}, looks as if it cannot avoid assuming a separated dependence on the state and control variables, since this is the structure of the variational problem for which the existence of solution has been so far ensured (\cite{CC}). 

Concerning the classical Filippov-Roxin theory introduced in \cite{F} and \cite{Rx}, it is not easy at all to know if typical convexity assumptions hold, or when they may hold, as we can see from the examples and counter-examples in \cite{C1}. When analyzing explicit examples, one realizes such difficulties coming from the need of a deep understanding of typical orientor fields. The same troubles would arise when applying refinements of this result as the ones in \cite{Mk} and \cite{MP}. 

 Recently (\cite{CFM}), an existence result has been shown for minimum time problems where the typical convexity assumptions over the set valued function on the differential inclusion has been replaced by more general conditions. In fact, the intersection of this result with the ones we present here is not empty although, as we will comment, our frame extends to situations not covered by this result. Such analysis can be done by writing problem $(P_1)$ as a minimum time problem as suggested in \cite{C1}. 


Our aim is to provide hypotheses on the different ingredients of the problem so that existence of solutions can be achieved through an independent road. Actually, it is not easy to claim whether our results improve on classical or more recent general results. They provide an alternative tool which can be more easily used in practice than such results when one faces an optimal control problem under the special structure we consider here. As a matter of fact, convexity will also occur in our statements but in an unexpected and non-standard way. 

Before stating our main general result, a bit of notation is convenient. 
We will write
\begin{equation}\label{Def_functions}
c:\R^N\to\R^s,\quad \phi:\R^n\to\R^s,\quad Q:\R^N\to\R^{Ns},
\end{equation}
with components  $c_i$, $\phi_i$, and $Q_i$, respectively. Consider also a new ingredient of the problem related to $\phi$. Suppose that there is a $\cC^1$ mapping

\begin{equation}\label{Def_function_Psi}
\Psi:\R^s\to\R^{s-n},\quad \Psi=(\psi_1,...,\psi_{s-n}), \quad (s>n),
\end{equation}
so that $\phi(K)\subset\{\Psi=0\}$. This is simply saying, in a rough way, that the embedded (parametrized) manifold $\phi(K)$ of $\R^s$ is part of the manifold defined implicitly by $\Psi=0$. In practical terms, it suffices to check that the composition $\Psi(\phi(u))=0$ for $u\in K$. 

For a pair $(c, Q)$, put
\begin{equation}\label{Def_NcQ}
\cN(c, Q)=\left\{v\in\R^s: Qv=0, cv\le0\right\}.
\end{equation}

Similarly, set
\begin{equation}\label{Def_NKphi}
 \cN(K, \phi)=
\end{equation}
$$
\left\{v\in\R^s: \hbox{ for each }u\in K,\hbox{ either }\nabla\Psi(\phi(u))v=0\hbox{ or there is $i$ with }\nabla\psi_i(\phi(u))v>0\right\}.
$$
Our main general result is the following.

\begin{theorem}\label{main}
Assume that the mapping $\Psi$ as above is strictly convex (componentwise) and $\cC^1$. If for each $x\in\R^N$, we have
\begin{equation}\label{Condition_inTeoMain}
\cN(c(x), Q(x))\subset\cN(K, \phi),
\end{equation}
then the corresponding optimal control problem $(P_1)$ admits at least one solution.
\end{theorem}

As it stands, this result looks rather abstract, and it is hard to grasp to what extent may be applied in more specific situations. 

A particular, yet still under some generality, situation where this result can be implemented is the case of polynomial dependence where the $\phi_i$'s are polynomials of various degrees. The main structural assumption, in addition to the one coming from the set $K$, is concerned with the convexity of the corresponding mapping $\Psi$. 

Suppose we take $\phi_i(u)=u_i$, for $i=1, 2, \dots, n$, and
$\phi_{n+i}(u)$, $i=1, 2, \dots, s-n$, convex polynomials of whatever degree, or simply polynomials whose restriction to $K$ is convex. In particular, $K$ itself is supposed to be convex. Then we can take
\begin{equation}\label{Def_psi_n}
\Psi_i(v)=\phi_{n+i}(\overline v)-v_{n+i},\quad i=1, 2, \dots, s-n, \quad \overline v=(v_i)_{i=1, 2, \dots, n}.
\end{equation}

In this case, it is clear that 
$$
\Psi(\phi(u))=0\hbox{ for }u\in K,
$$
by construction, and, in addition, $\Psi$ is smooth and convex. The important constraint (\ref{Condition_inTeoMain}) can also be analyzed in more concrete terms, if we specify in a better way the structure of the problem. 

As an illustration, though more general results are possible, we will concentrate on an optimal control problem of the type
\begin{description}
 \item[$(P)$]
\begin{equation}\label{vectorial1}
\hbox{Minimize in }u:\quad \int_0^T\left[ \sum_{i=1}^n c_i(x(t))u_i(t)+\sum_{i=1}^nc_{n+i}(x(t))u_i^2(t)\right]\,dt
\end{equation}
subject to
\begin{equation}\label{vectorial2}
x'(t)=Q_0(x(t))+Q_1(x(t))u(t)+Q_2(x(t))u^2(t)\hbox{ in }(0, T),
\end{equation}
\begin{equation}\label{vectorial3}
x(0)=x_0\in\R^n, \text{ and } u(t)\in K\subset\R^n. 
\end{equation}
\end{description}

We are taking here $N=n$. 
$Q_1$ and $Q_2$ are $n\times n$ matrices that, together with the vector $Q_0$, comply with appropriate technical hypotheses so that the state law is a well-posed problem. Set
\begin{equation}\label{Def_Q}
Q=\begin{pmatrix}Q_1&Q_2\end{pmatrix},\quad c=\begin{pmatrix}c_1&c_2\end{pmatrix},
\end{equation}
where $Q_1$ is a non-singular $n\times n$ matrix, and $c_1\in\R^n$. In addition, we put
\begin{equation}\label{Def_DEU}
D(x)=-(Q_1)^{-1}Q_2,\quad E(x)=c_1D+c_2,\quad U(m, x)=2\sum_im_ie_i\otimes e_i D-\hbox{id},\quad m=\phi(u),
\end{equation}
where the $e_i$'s stand for the vectors of the canonical basis of $\R^n$, and id is the identity matrix of size $n\times n$. 

\begin{theorem}\label{sec}
Suppose that for the ingredients $(c, Q, K)$ of $(P)$, we have
\begin{enumerate}
\item the matrix $U$ is always non-singular for $u\in K$, and $x\in\R^{n}$;
\item for such pairs $(u, x)$, we always have $U^{-T}E<0$, componentwise. 
\end{enumerate}
Then the optimal control problem admits solutions. 
\end{theorem}
 
As a more specific example of the kind of existence result that can be obtained through this approach, we state the following corollary whose proof amounts to going carefully through the arithmetic involved after Theorem \ref{sec}. 

\begin{corollary}\label{ter}
Consider the optimal control problem 
$$
\hbox{Minimize in }u:\quad \int_0^T [c_1(x(t))(u_1(t))^2+c_2(x(t))u_2(t)^2]\,dt
$$
under
\begin{gather}
x_1'(t)=u_1(t)-u_2(t)+q_{1}(x)u_1(t)^2+u_2(t)^2\nonumber,\\
x_2'(t)=q_2(x)u_1(t)+u_2(t)+u_1(t)^2+u_2(t)^2,\nonumber
\end{gather}
and an initial condition $x(0)=x_0$, 
where $u(t)\in K=[0,1]^2$,
$$q_1(x)\in(\frac{1}{3},1),\quad q_2(x)\in(-1,1),\quad  c_1(x),c_2(x)>0,$$ and

$$\frac{q_1(x)+1}{2}c_2(x)<c_1(x)<\frac{2(q_1(x))^2+q_1(x)(q_2(x)+1)-q_2(x)-3}{4(q_1(x)-1)}c_2(x).$$
Then there is, at least, one optimal solution of the problem.
\end{corollary}

Our strategy to prove these results is not new as it is based on the well-established philosophy of relying on relaxed versions of the original problem, and then, under suitable assumptions, prove that there are solutions of the relaxed problem which are indeed solutions of the original one (\cite{Ck}, \cite{G}, \cite{Ms}, \cite{Ms1}, \cite{W} and \cite{Y}). From this perspective, it is a very good example of the power of relaxed versions in optimization problems.

The relaxed version of the problem that we will be using is formulated in terms of Young measures associated with sequences of admissible controls. These so-called parametrized measures where introduced by
L. C. Young (\cite{Y}, \cite{Y1} and \cite{Y2}), and have been extensively used in Calculus of Variations and Optimal Control Theory
 ( see for example \cite{MP}, \cite{P1}, \cite{Rbk1} and \cite{Rbk3}). Because of the special structure of the dependence on $u$, we will be concerned with (generalized) ``moments" of such probability measures. Namely, the set
\begin{equation}\label{Def_L}
L=\left\{ m\in\R^s: m_i=\phi_i(u), 1\le i\le s, u\in K\right\}, 
\end{equation}
and the space of moments
\begin{equation}\label{Def_Lambda}
\Lambda=\left\{ m\in\R^s: m_i=\int_K\phi_i(\lambda)\,d\mu(\lambda), 1\le i\le s, \mu\in P(K)\right\}
\end{equation}
will play a fundamental role.  Here $P(K)$ is the convex set of all probability measures supported in $K$. Since the mapping
$$
M:\mu\in P(K)\mapsto\Lambda,\quad M(\mu)=\int_K\phi(\lambda)\,d\mu(\lambda)
$$
is linear, we easily conclude that $\Lambda$ is a convex set of vectors, and, in addition, that the set of its extreme points is contained in $L$. In fact, for some particular $\phi_i$'s of polynomial type, the set of the extreme points of $\Lambda$ is precisely $L$. We examine and comment on the set $\Lambda$ in Section~\ref{Sec_Lambda_duality}. This is closely related to the classical moment problem (\cite{A}, \cite{ST} or more recently \cite{EMP}, \cite{Mz}).

A crucial fact in our strategy is the following.  

\begin{assumption}\label{princ} For each fixed $x\in\R^N$, and  $\xi\in Q(x)\Lambda\subset\R^N$, the minimum
$$
\min_{m\in\Lambda}\left\{ c(x)\cdot m: \xi=Q(x)m\right\}
$$
is only attained in $L$.
\end{assumption}

It is interesting to realize the meaning of this assumption. If we drop the linear constraint $\xi=Qm$ on the above minimum, then the minimum is always attained in a certain point in $L$ simply because a linear function on a convex set will always take its extreme values on extreme points of such convex set. However, precisely the presence of the linear constraint $\xi=Qm$ makes the hypothesis meaningful as the extreme points of the section of $\Lambda$ by such set of linear constraints may not (indeed most of the time they do not) belong to $L$, so that the extreme points of the linear function $c\cdot m$ over such convex section may not attain its minimum on $L$. Our main hypothesis establishes that this should be so, and fundamentally, that the minimum is only attained in $L$.

Under this assumption, and the other technical requirements indicated at the beginning, one can show a general existence theorem of optimal solutions for our problem.

\begin{theorem}\label{gen}
Under Assumption \ref{princ} and the additional well-posedness hypotheses on $(c,Q)$ indicated above, the initial optimal control problem $(P_1)$ admits a solution.
\end{theorem}

Notice that we are not assuming any convexity on the set $K$ in this statement. The proof of this theorem can be found in Section~\ref{Sec_proof_of_theor}. As remarked before, the proof is more-or-less standard, and it involves the use of an appropriate relaxed formulation of the problem in terms of moments of Young measures (\cite{MP}, \cite{Rbk3}).

Condition (\ref{Condition_inTeoMain}) in Theorem \ref{main} is nothing but a sufficient condition to ensure Assumption \ref{princ} in a more explicit way. As a matter of fact, all of our efforts are directed towards finding in various ways more explicit conditions for the validity of this assumption. In this vein, 
the rest of the paper focuses on exploring more fully our Assumption \ref{princ} either through duality, geometric arguments, or in order to prove Theorem \ref{main}. Ideally, one would like to provide explicit results saying that for a certain set $\cM$, Assumption \ref{princ} holds if for each $x\in\R^N$, $(c(x), Q(x))\in\cM$. 
In fact, by looking at Assumption \ref{princ} from the point of view of duality, one can write a general statement whose proof is a standard exercise.

\begin{proposition}\label{propdual}
If for any $x\in \R^N$, $(c, Q)=(c(x),Q(x))$ are such that  for every $\eta\in\R^N$ there is a unique $m(\eta)\in L$ solution of the problem
\begin{equation}\label{Condition_inProp}
\hbox{Minimize in }m\in L:\quad (c+\eta Q)m
\end{equation}
then Assumption \ref{princ} holds.
\end{proposition}

We briefly comment on this in Section~\ref{Sec_Lambda_duality}.
One then says that $(c, Q)\in\cM$ if this pair verifies the condition on this proposition.  
A full analysis of this set $\cM$ turns out to depend dramatically 
on the ingredients of the problem. In particular, we will treat the cases $n=N=1$, and the typical situation of algebraic moments of degree 2 and 3 in Section~\ref{Sec_p=2}, Section~\ref{Sec_p=3}, and Section~\ref{Sec_Yet_p=3}. In Section~\ref{Sec_Examples} we apply our results to a few explicit examples and compare it with the application of the classical Filippov-Roxin theory.

Situations where either $N>1$ or $n>1$ are much harder to deal with, specially because existence results are more demanding on the structure of the underlying problem. In particular, we need a convexity assumption on how the non-linear dependence on controls occurs. We found that  (\ref{Condition_inTeoMain}) turns out to be a general sufficient condition for the validity of Assumption \ref{princ}, thus permitting to prove Theorem \ref{main} based on Theorem \ref{gen}. Both Theorem \ref{sec}, and Corollary \ref{ter} follow then directly from Theorem \ref{princ} after some algebra. This can be found in Section~\ref{Sec_Multi_Dim}. 

Finally, we would like to point out that one particular interesting example, from the point of view of applications, that adapts to our results comes from the control of underwater vehicles (submarines). See \cite{Bm}, \cite{Fs}, and \cite{HL}. This served as a clear motivation for our work. We plan to go back to this problem in the near future. 

\section{Proof of Theorem \ref{gen}}\label{Sec_proof_of_theor}
Consider the following four formulations of the same underlying optimal control problem. 
\begin{description}
 \item[$(P_1)$] The original optimal control problem described in (\ref{original1})-(\ref{original3}).
\item[$(P_2)$] The relaxed formulation in terms of Young measures  (\cite{MP}, \cite{P1}, \cite{Rbk1}, \cite{Rbk3}) associated with sequences of admissible controls:
$$
\hbox{Minimize in }\mu=\{\mu_t\}_{t\in(0, T)}:\quad \tilde{I}(\mu)=\int_0^T [\int_K\sum_{ i}c_{i}(x(t))\phi_i(\lambda)\,d\mu_t(\lambda)]\,dt
$$
subject to
$$
x'(t)=\int_K\sum_{i}Q_{i}(x(t))\phi_i(\lambda)\,d\mu_t(\lambda)
$$
and
$$
\hbox{supp}(\mu_t)\subset K, \quad x(0)=x_0\in\R^N.
$$
\item[$(P_3)$] The above relaxed formulation $(P_2)$ rewritten by taking advantage of the moment structure of the cost density and the state equation. If we put 
$c=(c_1,...,c_s)\in\R^s$, $Q\in M_{N\times s}$  and $m$ such that  
$$m_i=\int_K \phi_i(\lambda)\,d\mu_t(\lambda)\ \forall i\in\{1,...,s\},$$
then we pretend to
$$
\hbox{Minimize in }m\in\Lambda:\quad \int_0^T c(x(t))\cdot m(t)\, dt 
$$
subject to 
$$
x'(t)=Q(x(t))m(t),\quad x(0)=x_0.
$$
\item[$(P_4)$] Variational reformulation  of formulation $(P_3)$ (\cite{C1}, \cite{P}, \cite{PT}, \cite{R1}). This amounts to defining an appropriate density by setting
$$
\varphi(x,\xi)=\min_{m\in \Lambda}\{c(x)\cdot m \ : \xi=Q(x)m\}.
$$
Then we would like to 
$$
\hbox{Minimize in }x(t):\quad \int_0^T\varphi(x(t), x'(t))\,dt
$$
subject to $x(t)$ being Lipschitz in $(0, T)$ and $x(0)=x_0$. 
\end{description}

We know that the three versions of the problem $(P_2)$, $(P_3)$, and $(P_4)$ admit solutions because they are relaxations of the original problem $(P_1)$. In fact, since $K$ is compact, $(P_2)$ is a particular case of the relaxed problems studied in \cite{MP} and \cite{Rbk3}. The existence of solution for the linear optimal control problem $(P_3)$ is part of the classical theory (\cite{C1}). Indeed, $(P_3)$ is nothing but $(P_2)$ rewritten in terms of moments, so that the equivalence is immediate. $(P_4)$ is the reformulated problem introduced in \cite{R1} whose equivalence to $(P_3)$ was largely explored in \cite{C1} and \cite{P}, \cite{PT}. 

Let $\tilde x$ be one such solution of $(P_4)$. By Assumption \ref{princ} applied to a. e. $t\in(0, T)$, we have
$$
\varphi(\tilde{x}(t),\tilde{x}'(t))=\min_{m\in\Lambda}\{c(\tilde{x}(t))\cdot m(t): \  \tilde{x}'(t)=Q(\tilde{x}(t))m(t)\}=c(\tilde{x}(t))\cdot\tilde{m}(t)
$$
for a measurable  $\tilde{m}(t)\in L$, a solution of $(P_3)$ (see \cite{P}). The fundamental fact here (through Assumption \ref{princ}) is that $\tilde m(t)\in L$ for a.e. $t\in(0, T)$, and this in turn implies that $\tilde m(t)$ is the vector of moments of an optimal Dirac-type Young measure $\mu=\{\mu_t\}_{t\in(0, T)}=\{\delta_{\tilde u(t)}\}_{t\in(0, T)}$ for an admissible $\tilde u$ for $(P_1)$. This admissible control $\tilde u$ is optimal for $(P_1)$. This finishes the proof.

\section{The set $\Lambda$ and duality}\label{Sec_Lambda_duality}
The moment set $\Lambda$ deserves some comments before proceeding further. Consider the mapping 
$\phi$ as in (\ref{Def_functions}) and $L$ as in (\ref{Def_L}).

We can regard $L$ as part of an embedded $n$-manifold in $\R^s$, $s>n$, and $\phi$ its standard or canonical parametrization. The moment set $\Lambda$ defined in (\ref{Def_Lambda}) is contained in the convex hull of this manifold.

The most important fact about $\Lambda$ that one may need in our analysis is stated in the next proposition. 

\begin{proposition}\label{L_extLambda}
The set of extreme points of $\Lambda$ is contained in $L$.
\end{proposition}

\begin{proof}
 First notice that, as it was shown in \cite{Mz} in a context similar to ours, the compactness of $K$ implies  
$$\overline{co (L)}=co (L)=\bar{\Lambda}=\Lambda.$$
The fact of $K$ being bounded plays an important role because otherwise $\Lambda$ can be shown to be not necessarily closed (\cite{EMP}).

 Since $\Lambda=co (L)$ then it is known from convex analysis (\cite{R}) that  
$$ext (\Lambda)\subseteq L,$$
where $ext (\Lambda)$ represents the extreme points of $\Lambda$.
\end{proof}

\begin{remark}
 For some $\phi$'s it is possible to conclude that 
$ext (\Lambda)=L$. This is the case, for example when $\phi$ contains all the linear and quadratic terms of a $n$-variable polynomial. However this is not essential in what follows.
\end{remark}

Due to this result the proof of Proposition~\ref{propdual} is standard (see \cite{R}), so that we shall only make a few remarks.

Since
$$
ext (\Lambda)\subseteq co (L)
$$
 which is a compact set, the minimum of 
$$(c+\eta Q)m$$ 
in $\Lambda$ is always attained at least in one point of $L$ (it can be attained also in points of $\Lambda\setminus L$). However, if this point happens to be unique, because of Proposition~\ref{L_extLambda}, it is also immediate to check that it must be the unique minimizer in $\Lambda$.

The condition (\ref{Condition_inProp}) in Proposition~\ref{propdual} means that

$$\min_{m\in \Lambda} (c+\eta Q)m= \min_{m\in L} (c+\eta Q)m=(c+\eta Q)\phi(a)$$ for a single $a\in K$, which also verifies
$$\min_{m\in \Lambda} (c+\eta Q)m-\eta\xi= (c+\eta Q)\phi(a)-\eta \xi$$
for $\xi\in Q(x)\Lambda$, that is, such that Assumption~\ref{princ} is non empty.

In particular the associated Karush-Kuhn-Tucker vector $\bar{\eta}$ verifies (see \cite{R})

$$c\cdot\phi(a)+\bar{\eta}(Q\phi(a)-\xi)=\min_{m\in\Lambda}\{c\cdot m:\ Qm=\xi\}=c\cdot \phi(a)$$ for a single $a\in K$ complying with $Q\phi(a)=\xi$. As a consequence, for all admissible $m\in \Lambda$ different from $\phi(a)$, we have 
$$ c\cdot m>c \cdot \phi(a).$$ 

\section{Polynomial Dependence. The case  $N=n=1$, $p=2$}\label{Sec_p=2}

Until Section 8, we concentrate in the situation where 
$$\phi:\R^n\to\R^s$$ is such that 
$\phi_i(u)=u_i$, for $i=1, 2, \dots, n$, and 
$\phi_{n+i}(u)$, $i=1, 2, \dots, s-n$ are convex polynomials of some degree $p$, or simply polynomials whose restriction to $K$ is convex.
We will consider $K$ itself to be convex.

Our goal is to explore different possibilities to apply directly Theorem~\ref{gen} by ensuring Assumption~\ref{princ}. In other words, we will search for functions 
$$
c:\R^N\to\R^s,\quad  Q:\R^N\to\R^{Ns},
$$
such that for every $x\in \R^N$,
$$(c(x),Q(x))\in \cM$$ where $\cM$ represents the set
\begin{equation}\label{Def_M}
\left\{(c,Q):\ \forall\ \xi\in Q\Lambda, \ 
\arg\min_{m\in\Lambda}\{ c\cdot m: \xi=Qm\}\in L\right\} 
\end{equation}

During the following three sections we will focus on the one dimensional case $N=n=1$ and use some ideas based in duality (Proposition~\ref{propdual}) and in geometric interpretations.

In Sections 4, 5, and 6, we explore various scenarios where Assumption \ref{princ} can be derived, and defer explicit examples until Section 7. In particular, we consider in this section the situation where $\phi$ is given by $\phi(a)=(a,a^2)$. We are talking about polynomial components of degree less or equal than $p=2$.

Let $K=[a_1,a_2]$, $L$, and $\Lambda$ as in (\ref{Def_L})-(\ref{Def_Lambda}). Here, we have $s=2$ and 
$$
c:\R\to\R^2,\quad Q:\R\to\R^2
$$
can be identified with vectors in $\R^2$, or more precisely, with plane curves parametrized by $x$. To emphasize that function $Q$ is not a matrix-valued but vector-valued, we will call it $q$.

\vspace{0.5 cm}

\begin{figure}[h]
\centering  
\includegraphics[scale=0.4,natwidth=3cm, natheight=4cm]{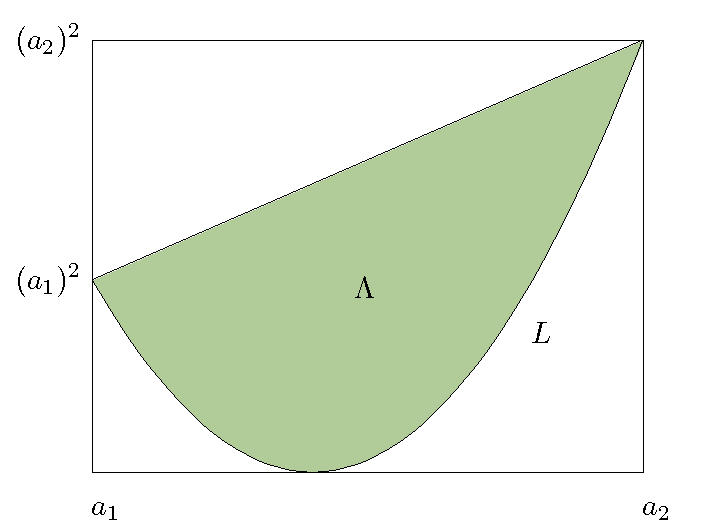} 
 \caption{$\Lambda=co(L)$ for $p=2$}
\end{figure}

\vspace{0.5 cm}
Next we describe sufficient conditions for $(c(x),q(x))\in \cM$.
\begin{lemma}\label{LemaDualidadeGrau2}
Let $K$, $L$ and $\phi$ be as above. For every $x\in \R$, let $q=q(x)$ and $c=c(x)$ be vectors such that one of the following conditions is verified  
\begin{enumerate}
\item $q_1+q_2(a_1+a_2)=0$ and 
$$\hbox{det}\begin{pmatrix} c_1&c_2\\ q_1&q_2\end{pmatrix}\neq 0;$$
\item $q_1+q_2(a_1+a_2)\neq0$ and
$$(q_1+q_2(a_1+a_2))\hbox{det}\begin{pmatrix} c_1&c_2\\ q_1&q_2\end{pmatrix}<0.$$
\end{enumerate}
Then $(c, q)\in\cM$,
 and consequently Assumption~\ref{princ} is verified.
\end{lemma}

\begin{proof}
 Suppose there is $\eta$ such that the minimum of $(c+\eta q)\cdot m$ is attained in more than one point of $L=\phi(K)$. This means that the real function
$$
g(t)=(c+\eta q)\cdot\phi(t)=(c_1+\eta q_1)t+(c_2+\eta q_2)t^2
$$
has more than one minimum point over $K$. For that to happen, either $g$ is constant on $t$, i. e.,
$$
\begin{cases}
 c_1+\eta q_1=0\\
c_2+\eta q_2=0
\end{cases}
\Leftrightarrow \hbox{det}\begin{pmatrix} c_1&c_2\\ q_1&q_2\end{pmatrix}=0,
$$
which contradicts our hypothesis; or else we must have
\begin{equation*}
c_2+\eta q_2<0,\quad g'\left(\frac{a_1+a_2}{2} \right)=0.
\end{equation*}
This condition can be written as
$$
c_1+(a_1+a_2)c_2+\eta[q_1 + (a_1+a_2)q_2]=0.
$$
If $q_1+q_2(a_1+a_2)=0$, but $c_1+(a_1+a_2)c_2\neq0$ (condition 1. in statement of lemma), then this equation can never be fulfilled. Otherwise, there is a unique value for $\eta$, by solving this equation, which should also verify the condition on the sign of $c_2+\eta q_2$. It is elementary, after going through the algebra, that the condition on this sign cannot be true under the second condition on the statement of the lemma. 
\end{proof}


\section{The case $N=n=1$, $p=3$}\label{Sec_p=3}

We study the case where $\phi(a)=(a,a^2,a^3)$, $s=3$, and $c$ and $q$ can be identified as vectors in $\R^3$. The understanding of the set $\Lambda$ and its sections by planes in $\R^3$ is much more subtle however. 

\begin{figure}[h]
\centering
\includegraphics[scale=0.4,natwidth=3cm, natheight=4cm]{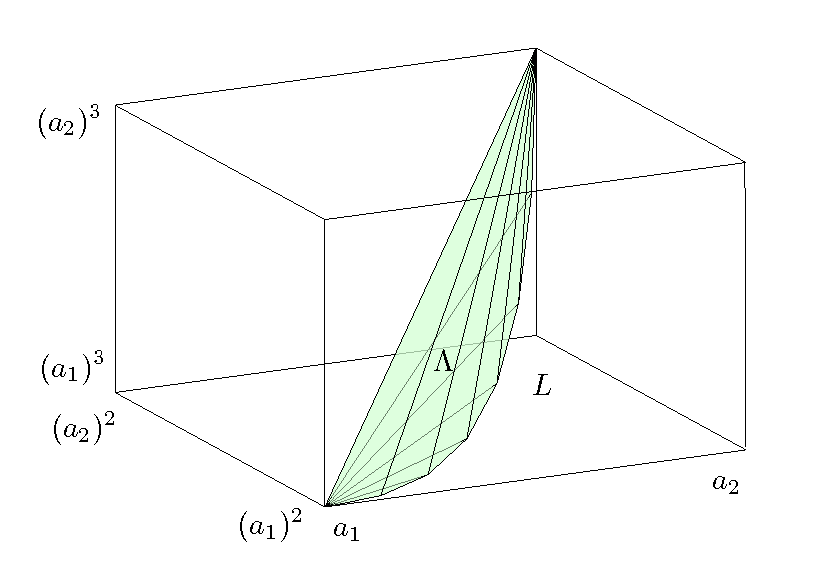}
 \caption{$\Lambda=co(L)$ for $p=3$}
\end{figure}

To repeat the procedure used for $p=2$, and apply Proposition~\ref{propdual}, we would like to give sufficient conditions for the function
\begin{equation}\label{cubic}
g(t)=(c+\eta q)\cdot \phi (t)=(c_1+\eta q_1)t+(c_2+\eta q_2)t^2+(c_3+\eta q_3)t^3
\end{equation}
to have a single minimum over $K=[a_1,a_2]$ for every $\eta$. As indicated, and after some reflection, a complete analysis of the situation is rather confusing and the conditions on the vectors $c$ and $q$ much more involved. To illustrate this, we give a sufficient condition in the following form.

\begin{lemma}
For all $x\in \R$, let $c=c(x)$ and $q=q(x)$ be vectors in $\R^3$ such that 
$$
q_2^2-3q_1q_3<0,\quad
(2c_2q_2-3c_1q_3-3q_1c_3)^2-4(c_2^2-3c_1c_3)(q_2^2-3q_1q_3)<0,
$$
then $(c, q)\in\cM$,
 and Assumption~\ref{princ} is verified.
\end{lemma}

\begin{proof}
The proof consists in the realization that the conditions on the vectors $c$ and $q$ ensure that the cubic polynomial (\ref{cubic}) is monotone in all of $\R$ (avoiding degenerate situations), and thus it can only attain the minimum in a single point of any finite interval. Notice that this condition is independent of the interval. In fact, we have to discard the possibility for the derivative of the polynomial $g(t)$ to have roots. This amounts to the negativity of the corresponding discriminant. And this, in turn, is a quadratic expression in $\eta$ that ought to be always negative. This occurs when that parabola has a negative discriminant, and the leading coefficient is also negative. These two conditions are exactly the ones in the statement of this lemma.
\end{proof}

A more general condition would focus on considering the local maximizer and the local minimizer of $g(t)$, $M_+$ and $M_-$, respectively, and demanding that the interval $[a_1, a_2]$ have an empty intersection with the interval determined by $M_+$ and $M_-$. But this would lead to rather complicated expressions. Even so, some times under more specific hypotheses on the form of the vectors $c$ and $q$, these conditions can be exploited. 

\begin{remark}
Notice that the relation 
$$
ext (\Lambda) = L
$$
is not true for a general $K$ if it has positive and negative values. However, it is true if we consider $a_1>0$ or $a_2<0$. 
\end{remark}

\begin{lemma}\label{LemaDualidadeGrau3}
 Let $K=[a_1,a_2]$ with $a_1>0$ and 
$$
(c,q)=((0,c_2,c_3),(0,q_2,q_3))
$$ 
such that 

\begin{equation*}
-\frac{q_2}{q_3}<0,\quad
(c_2,c_3)\cdot(1,-\frac{q_2}{q_3})<0.
\end{equation*}

 Then the assumptions of Proposition~\ref{propdual} are valid, and consequently so is Assumption~\ref{princ}.
\end{lemma}

\begin{proof}

In this situation, the maximizer $M_+$ referred to above is given by

$$
M_+=\dfrac{-(c_2+\eta q_2)-\vert c_2+\eta q_2\vert}{3(c_3+\eta q_3)}
$$
If $q_2>0$, then $q_3>0$, and we have 
$$
-\frac{c_2}{q_2}> -\frac{c_3}{q_3}.
$$
Hence if $\eta\in ]-\infty,-\frac{c_2}{q_2}]\setminus\{-\frac{c_3}{q_3}\}$
$$
M_+=\dfrac{-(c_2+\eta q_2)+ c_2+\eta q_2}{3(c_3+\eta q_3)}=0.
$$

If $\eta>-\frac{c_2}{q_2}$,
$$
M_+(\eta)=\dfrac{-(c_2+\eta q_2)- (c_2+\eta q_2)}{3(c_2+\eta q_2)}=\dfrac{-2(c_2+\eta q_2)}{3(c_3+\eta q_3)}<0.
$$
In any case 
$M_+(\eta)\leq 0$, thus $a_1>M_+$.

Also if $\eta=-\frac{c_3}{q_3}$, 
$$
g(t)=(c_2+\eta q_2)t^2=(c_2,c_3)\cdot(1,-\frac{q_2}{q_3})t^2
$$ 
which has a unique minimum in $K$ since we have assumed $a_1>0$.
We conclude that the condition (\ref{Condition_inProp}) in Proposition~\ref{propdual} is verified.
\end{proof}

In a very similar way we can prove the following.

\begin{lemma}
 Let $K=[a_1,a_2]$ with $a_2<0$ and 
$$
(c,q)=((0,c_2,c_3),(0,q_2,q_3))
$$ 
such that 
\begin{equation*}
-\frac{q_2}{q_3}>0,\quad
(c_2,c_3)\cdot(1,-\frac{q_2}{q_3})<0.
\end{equation*}
 Then $(c,q)\in\cM$ and consequently Assumption~\ref{princ} is valid.
\end{lemma}

\section{A geometric approach to the case $N=n=1$, $p=3$.}\label{Sec_Yet_p=3}

As we have seen, the use of Proposition~\ref{propdual} is simpler only when restricted to some particular classes of examples. Thus we propose a general criteria for obtaining Assumption~\ref{princ}, based on a geometric approach.

We first give a result that generalizes the strictly convexity of a $\phi$-parametrized plane curve for a 3-dimensional one.

\begin{lemma}\label{Lemma ConvexidadeEstrita p=3}
Let $K=[a_1,a_2]$ with $a_1>0$, $\phi(t)=(t,t^2,t^3)$, and $L$ the curve parametrized by $\phi$ for $t$ in $K$. 
\begin{enumerate}
\item 
Given $t$ in $K$, then for all $s\in K$ such that $s\neq t$ we have 
$$(\phi(s)-\phi(t))\cdot N(t)>0$$ where $N(t)$ is the normal vector to $\phi$ at $t$.
\item 
For every $t \in K, \ v\in \Lambda=co(L)\setminus\{\phi(t)\}$, we have
$$(v-\phi(t))\cdot N(t)>0.$$
\end{enumerate}
\end{lemma}

\begin{proof}

\text{ To check the first part of the statement notice that }
since 
$$\phi'(t)=(1,2t,3t^2)$$ 
and 
$$\phi''(t)=(0,2,6t)$$ we have that the normal vector, colinear to $\phi'(t)\times \phi''(t)$, is given by

$$N(t)=c_t(-9t^2-2t,1-9t^4,6t^3+3t)$$ where $c_t>0$ is a normalizing constant.
Setting 
$$N_1=-9t^3-2t,\quad N_2=1-9t^4,\quad N_3=6t^3+3t,
$$
we find that the solution $s$ of 
$$(\phi(s)-\phi(t))\cdot N(t)=0$$ also verifies
$$N_3s^3+N_2s^2+N_1s-N\cdot\phi(t)=0,$$ which is equivalent to
$$(s-t)^2(N_3s+N_2+2tN_3)=0.$$ This means that 
$$\hat{s}=-\frac{N_2}{N_3}-2t=-\frac{3t^4+6t^2+1}{6t^3+3t}$$ is the only solution different from $t$, but also that it is negative for all $t>0$, and consequently that it should be excluded.
Once we assumed $K\subset \R^{+}$ and $s\neq t$ the conclusion is immediate. 

By using the previous discussion, proving the second part of the statement is trivial once we notice that both $m$ and $\phi(t)$ can be rewritten as
$$\sum_{i=1}^4\alpha_i\phi(s_i) \text{ and } \sum_{i=1}^4\alpha_i\phi(t)$$ respectively, where $s_i\in K$ and $\sum_{i=1}^4\alpha_i=1$.

\end{proof}

Another useful lemma.

\begin{lemma}\label{Lemma ConsequenciadeCondUnicidade}
 If $q$ and $c$ are such that 
\begin{equation}\label{Cond_p3}
(\phi'(t)\times (c\times q)) \cdot(\phi(s)-\phi(t))
\end{equation}
does not change sign for $t,s\in K,\ s\neq t $, then if $v\in \Lambda$, $v\neq \phi(t)$, and $q\cdot(v-\phi(t))=0$ we have 
$$c\cdot(v-\phi(t))\neq 0.$$
\end{lemma}

This means that the linear function $c$ cannot take the same value over $\phi(t)$ and any $v\neq\phi(t)$ in the plane section
$$\{v\in \Lambda:\ q\cdot v=q\cdot \phi(t)\}.$$

\begin{proof}
Notice that for $v\in \Lambda$,
\begin{eqnarray*}
(\phi'(t)\times(c\times q)) \cdot(v-\phi(t)) &=&
(\phi'(t)\times(c\times q)) \cdot(\sum_{i=1}^4\alpha_i\phi(s_i)-\sum_{i=1}^4\alpha_i\phi(t))\\
&=& \sum_{i=1}^4\alpha_i(\phi'(t)\times(c\times q))\cdot(\phi(s_i)-\phi(t))>0 (\text{ or }<0),
\end{eqnarray*}
so that the condition stated is also verified for any $v\in \Lambda$.

Suppose now that $v\in \Lambda$ verifies   
$q\cdot(v-\phi(t))=0$ for given $t\in K$ with $v\neq \phi(t)$ and is such that $c\cdot(v-\phi(t))=0$, then 
\begin{eqnarray*}
(\phi'(t)\times(c\times q)) \cdot(v-\phi(t)) &=& [(\phi'(t)\cdot q)c-(\phi'(t)\cdot c)q]\cdot(v-\phi(t))\\
&=& (\phi'(t)\cdot q)c\cdot (v-\phi(t))-(\phi'(t)\cdot c)q\cdot(v-\phi(t))=0,
\end{eqnarray*}
a contradiction concerning the argument above.
\end{proof}

We now define the set 
$\cM_1$ of pairs $(c, q)\in\R^3\times\R^3$ through the following requirements: 

\begin{itemize}
 \item the quantity in (\ref{Cond_p3})
 does not change sign over the pairs $t,s\in K,\ s\neq t$;
\item whenever there is a unique $a\in K=[a_1,a_2]$ such that
\begin{equation}\label{Cond_p3PM}
 (\phi(a_1)+\phi(a_2)-2\phi(a))\cdot q=0,
\end{equation}
then
$$
(\phi(a_1)+\phi(a_2)-2\phi(a))\cdot c>0.
$$
\end{itemize}

Once more we can establish the following result.

\begin{proposition}\label{prop princ p=3}
Let $\cM$ be as in (\ref{Def_M}). 

If $a_1>0$, and $(c, q)\in\cM_1$, then $(c, q)\in \cM$ and Assumption~\ref{princ} holds.
\end{proposition}
\begin{proof}

1. Suppose first that there is $a\in K$ such that we have (\ref{Cond_p3PM}).

Let $$v_a=\frac{\phi(a_1)+\phi(a_2)}{2}.$$
Consider  $v\in \Lambda$ such that 
$$[v-\phi(a)]\cdot q=0.$$ 
Suppose 
$$c\cdot [v-\phi(a)]<0,$$ and consider the continuous function 
$$G(v,u)=c\cdot (v-u)$$ over the bounded path connecting $(v_a,\phi(a))$ and $(v, \phi(a)$ given by
$$S=\{\alpha[(v,\phi(a))-(v_a,\phi(a))]+(v_a,\phi(a)):\ \alpha\in[0,1]\}.$$
It is easy to check that every component of a vector of $S$ is contained in the section
$$\{v\in\Lambda: q\cdot v=q\cdot \phi(a)\}.$$

Then there exists $\alpha$ such that
$$G(\alpha[(v,\phi(a))-(v_a,\phi(a))]+(v_a,\phi(a)))=0,$$ or in other words
$$c\cdot [\alpha(v-v_a)+v_a-\phi(a)]=0,$$ which by Lemma~\ref{Lemma ConsequenciadeCondUnicidade} means that necessarily 
$$\alpha(v-v_a)+v_a=\phi(a).$$
Consequently 
$$\alpha(v-\phi(t))\cdot N(t)+(1-\alpha)(v_t-\phi(t))\cdot N(t)=0$$ and this is in contradiction with Lemma~\ref{Lemma ConvexidadeEstrita p=3}.
Hence
$$c\cdot [v-\phi(a)]>0 \text{ if }c\cdot [v_a-\phi(a)]>0.$$
Let $\bar{t}$ be such that
$$q\cdot \phi(a_1)=q\cdot \phi(\bar{t})$$
and $t\neq a$, $t\geq\bar{t}$, such that
$$v_t=\alpha[\phi(a_2)-\phi(a_1)]+\phi(a_1)\in \Lambda$$ verifies
$$[v_t-\phi(t)]\cdot q=0$$
Considering once more the continuous function $G(v,u)$ over the path connecting
$(v_t,\phi(t))$ and $(v_a,\phi(a))$, 
as
$$\alpha[v_t-v_a]+v_a\in \{v: q\cdot v=q\cdot \phi(t)\},$$
we can, as we did above, conclude that if 
$$c\cdot [v_t-\phi(t)]<0$$ then for certain $\alpha$, 
$$\alpha[v_t-v_a]+v_a=\phi$$ and consequently 
$$c\cdot [v_t-\phi(t)]>0$$ for any $t\geq \bar{t}$. The same type of arguments
show that 
$$c\cdot [v-\phi(t)]>0$$ for any $v$ such that
$$q\cdot v=q\cdot \phi(t).$$

If $t<\bar{t}$, there exists $s\in K$ such that 
$$q\cdot \phi(s)=q\cdot \phi(t).$$ In this situation,
again the continuity of $G$ should be applied to the path connecting
$$(v_{\bar{t}},\phi(\bar{t}))=(\phi(a_1),\phi(\bar{t}))$$ and
$$(\phi(s),\phi(\bar{t})),$$ repeatedly until the limit case when $\phi(s)=\phi(\bar{t})$.

If there is $\bar{t}\neq a_2$ such that
$$q\cdot \phi(a_2)=q\cdot \phi(\bar{t})$$ we shall proceed in an analogous way.

2. Suppose now that there are $a,b\in K$ such that
$$
(v_a-\phi(a))\cdot q=
(v_a-\phi(b))\cdot q=0.
$$
Then it is not difficult to conclude that
$$a=a_1 \text{ and } b=a_2.$$
Hence assuming (without loss of generality) that
$$(\phi(a_1)-\phi(a_2))\cdot c>0$$ we can, once again, use the continuity of $G$
to conclude
$$c\cdot [\phi(s)-\phi(t)]>0$$ where $\phi(s)$ and $\phi(t),$ verify
$$(\phi(s)-\phi(t))\cdot q=0$$ and after that, for a general $v$ such that
$$(v-\phi(b))\cdot q=0.$$

\end{proof}

\begin{remark}
 This type of argument can be also deduced for the case $N=n=1$, $p=2$ where it can be seen to be equivalent to the conditions in Lemma~\ref{LemaDualidadeGrau2}. However when the parameters $N$, $n$ and $p$ increase their values, it becomes very hard to give geometrically-based sufficient conditions in such an exhaustive manner as we have done here. Even so, in Section~\ref{Sec_Multi_Dim} we show how to give more restrictive yet more general sufficient conditions (Theorem~\ref{sec} and Theorem~\ref{main}) for interesting high dimensional particular situations, where some geometrical ideas can be used as a way to verify Assumption~\ref{princ}.
\end{remark}

\section{Examples}\label{Sec_Examples}
Before going further to higher dimensional situations we gather in this section some typical, academic examples for which either Lemma~\ref{LemaDualidadeGrau2}, Lemma~\ref{LemaDualidadeGrau3}, or Proposition~\ref{prop princ p=3} can be applied. 

\subsection{Example 1}
 
Let us consider the optimal control problem
$$
\hbox{Minimize in }u:\quad \int_0^T [c(x(t))u(t)+u^2(t)]\,dt
$$
under
$$
x'(t)=q(x(t))u(t)+u^2(t),\quad x(0)=x_0
$$
where $|u(t)|\le 1$. 

We have the following remarkable existence result. 

\begin{lemma}
If the functions $q$ and $c$ are Lipschitz, and 
$$
q(q-c)>0.
$$
then the optimal control problem admits solution.
\end{lemma}

The proof reduces to performing some elementary algebra to check the conditions of Lemma~\ref{LemaDualidadeGrau2}.


Instead of applying that lemma, as both our cost and state-equation functions have cross dependence on $x$ and on $u$ so that we can't apply results in \cite{Br}, \cite{Ry}, one can try the classical existence result based on the classical Filippov-Roxin theory.  For that we need to check if the orientor field  
$$\cA_x=\{(\xi,v):\ v\ge c(x)u+u^2,\ \xi=q(x)u+u^2,\ u\in K=[-1,1]\}$$ is a convex set. Notice that $K$ is bounded so coercivity is not an issue here.
Proceeding in that direction, we can  see that
$$\xi=q(x)u+u^2$$ is equivalent to
$$u_1=-\dfrac{q+\sqrt{q^2+4\xi}}{2}\ \text{or } u_2=-\dfrac{q-\sqrt{q^2+4\xi}}{2},$$ which are possible solutions when $\xi$ is such that $\xi\ge -\frac{q^2}{4}$, and
at least one of them belongs to $K=[-1,1]$. 
Letting
$$F_i(x,\xi)=c(x)u_i+u_i^2,\ i=1,2,$$
we see that $$F_2\leq F_1,$$ for all $\xi$ as above. 
Consequently
$$\cA_x=\cA_x^1\cup\cA_x^2=$$
$$\{(\xi,v):\ v\ge F_2(x,\xi), \xi\in u_2^{-1}(K)\cap[-\frac{q^2}{4},+\infty[\}\bigcup$$
$$\{(\xi,v):\ v\ge F_1(x,\xi), \xi\in \left(u_1^{-1}(K)\setminus u_2^{-1}(K)\right)\cap[-\frac{q^2}{4},+\infty[\}
$$ where, for $i=1,2$, $u_i^{-1}$ refers to the pre-image of the solutions $u_i$ as functions of $\xi$. 

Because of the assumption on $(c,q)$ it is easy to see that $\cA_x^2=\emptyset$, and consequently that the convexity of $\cA_x$ reduces to the convexity of the function 
$$F_2(\xi)=\dfrac{q-c}{2}(q-\sqrt{q^2+4\xi})$$ over a certain convex set $$u_2^{-1}(K)\cap[-\frac{q^2}{4},+\infty[.$$ This can be checked by elementary calculus.

We now turn over the possibility of applying the result in \cite{CFM} to this example. First, in order to write our problem as a minimum time problem, we need that $c(x)u+u^2$ never changes sign in $\R\times K$ (\cite{C1}). So a first restriction must be imposed. For example, consider $c(.)$ and $q(.)$ such that $$q(x)>c(x)>1.$$ The right member of the differential equation of the minimum time problem is given by 
$$\cF(x,K)=\{\dfrac{q(x)u+u^2}{c(x)u+u^2}:\ u\in K\}.$$
The result in \cite{CFM} doesn't ask for the convexity of the set-valued map  $\cF$,
 but it requires a linear boundedness in the sense that 
$$\exists \alpha, \ \beta \text{ s. t. }\forall x\in \R,\ \forall \xi \in \cF(x,K)\text{ then} $$
  $$\Vert\xi\Vert \leq \alpha \Vert x\Vert+\beta.$$ It is easy to see that this condition places a real constraint on the relative growth of pairs $(c, q)$, even before verifying the remaining assumptions in \cite{CFM}.

\subsection{Example 2}
Look at the problem
$$
\hbox{Minimize in }u:\quad \int_0^T [c(x(t))u^2(t)+u^3(t)]\,dt
$$
under
$$
x'(t)=[q(x(t))]u^2(t)+u^3(t),\quad x(0)=x_0
$$
where $u(t)\in [a_0,a_1]$, $a_0>0$. 

\begin{lemma}\label{Exemplo dualidade existencia p=3}
If the functions $q(x)$ and $c(x)$ are Lipschitz,  
$$
c(x)<q(x)\ \forall x,
$$ and $q(x)$ is always positive,
then the optimal control problem admits solutions.
\end{lemma}
This result comes directly by applying Lemma~\ref{LemaDualidadeGrau3} and Theorem~\ref{gen}.

Let us see, what we would need to do if, alternatively, we decided to use the classical existence theory.

Like we have seen in the previous example we need to check the convexity of the orientor field
$$\cA_x=\{(\xi,v):\ v\ge c(x)u^2+u^3,\ \xi=q(x)u^2+u^3,\ u\in K=[a_0,a_1]\}.$$
In this case, accordingly to the discriminant 
$$\Delta=27\xi^2-4\xi q$$ of the equation 
$$\xi=q(x)u^2+u^3$$ we will have from one to three possible real solutions. Consider for each $\xi$ 
$$F_i=c(x)u_i^2+u_i^3$$ 
such that 
$$F_1\leq F_2\leq F_3$$
where $u_i=u_i^x(\xi),\ i=1,2,3$ are the three, possible equal, real solutions.
Then
 
$$\cA_x=\cA_x^1\cup\cA_x^2\cup\cA_x^3=$$
$$\{(\xi,v):\ v\ge F_1,\  \xi\in u_1^{-1}(K)\}\bigcup$$
$$\{(\xi,v):\ v\ge F_2, \ \xi\in u_2^{-1}(K)\setminus u_1^{-1}(K)\}\bigcup$$
$$\{(\xi,v):\ v\ge F_3,\ \xi\in u_3^{-1}(K)\setminus \left (u_2^{-1}(K) \cup u_1^{-1}(K)\right)\}$$

Checking the convexity of this set, or alternatively, of the function
$$\varphi_x(\xi)=\begin{cases}
               F_1(\xi) \text{ if }  &\xi \in u_1^{-1}(K)\\
F_2(\xi) \text{ if }  &\xi\in u_2^{-1}(K)\setminus u_1^{-1}(K)\\
F_3(\xi) \text{ if }  &\xi\in u_3^{-1}(K)\setminus \left (u_2^{-1}(K) \cup u_1^{-1}(K)\right)
              \end{cases}
$$ 
is not an easy task at all, specially when compared to the almost immediate exercise of verifying the conditions of Lemma~\ref{LemaDualidadeGrau3}. It is also plausible that the inherent difficulties to apply classical theory will increase until a practically impossible scenario when we let $N$, $n$ and $p$ grow.

\subsection{Example 3}
In order to give an heuristic for using the criteria given in 
Proposition~\ref{prop princ p=3} let us consider the previous problem, just by rewriting $q$ as $c-\beta$ and for a specific $K$.
$$
\hbox{Minimize in }u:\quad \int_0^T [c(x(t))u^2(t)+u^3(t)]\,dt
$$
under
$$
x'(t)=[c(x(t))-\beta(x(t))]u^2(t)+u^3(t),\quad x(0)=x_0
$$
where $u(t)\in [1,2]$. 

\begin{lemma}\label{Exemplo existencia p=3}
If the functions $\beta$ and $c$ are Lipschitz, and 
$$
\beta<\min\{0,c\}.
$$
then the optimal control problem admits solutions.
\end{lemma}

\begin{proof} 
First notice that for $a\in K=[a_1,a_2]$, we can find $\alpha$ such that
the vector $$B=\alpha[\phi(a_2)-\phi(a_1)]+\phi(a_1)$$ verifies
$$[B-\phi(a)]\cdot q=0.$$ Moreover, it is not difficult to see that
$$\alpha=\frac{a^3-a_1^3-m(a^2-a_1^2)}{a_2^3-a_1^3-m(a_2^2-a_1^2)}$$
and in the projection plane $yz$, $(B_2,B_3)$ belongs to the line of slope $m$ passing through $(a^2,a^3)$,
$$B_3-a^3=m(B_2-a^2),$$
where
$$B_2-a^2=\frac{(a-a_1)[a^2(a_2+a_1)-a^2(a+a_1)]}{a_2^2+a_1a_2+a_1^2-m(a_2+a_1)}$$
and $m=-\frac{q_2}{q_3}.$

In our case $K=[1,2]$, so,
because of what we have just seen, taking $a_1=1$ and $a_2=2$ we see that for $a\in K$, we can find 
$$\alpha=\frac{a^3-ma^2+m-1}{7-3m}\in [0,1]$$ such that $$[\alpha[\phi(a_2)-\phi(a_1)]+\phi(a_1)-\phi(a)]\cdot q=0,$$
 where
$$m=-\frac{c-\beta}{1}=\beta-c<0.$$
Furthermore,
it is easy to see that the equation $\alpha=\frac{1}{2}$ has a unique solution in $K$. 
Consequently, if we consider
$q=(0,c-\beta,1)$ and 
$\bar{c}=(0,c,1)$, there exists a unique $a\in K$ such that  
$$\ [\phi(1)-\phi(0)-2\phi(a)]\cdot q=0.$$

Also, because of what we have seen above
\begin{eqnarray*}
[\frac{1}{2}(\phi(1)-\phi(0))-\phi(a)]\cdot \bar{c}&=
(B_2-a^2)c+(B_3-a^3)=(B_2-a^2)(c+m)\\
&=\frac{(a-1)(3a^2-4a-4)}{7-3m}(c+\beta-c)>0.
\end{eqnarray*}

In addition, given
 $t,s\in K,\ s\neq t$,
$$(\phi'(t)\times (c\times q)) \cdot(\phi(s)-\phi(t))=0\Leftrightarrow$$
$$\beta t(0,3t,-2)\cdot(\phi(s)-\phi(t))=0\Leftrightarrow$$
$$(s-t)[3t(s+t)-2(s^2+st+t^2)]=0\Leftrightarrow$$
$$s=-\frac{t}{2}\vee s=t$$ which is impossible since $s\in K=[1,2]$ and $s\neq t$.
The result follows then by applying Proposition~\ref{prop princ p=3}.
 
\end{proof}

\section{The case $N$, $n>1$}\label{Sec_Multi_Dim}
The previous analysis makes it very clear that checking  Assumption \ref{princ} may be a very hard task as soon as $n$ and/or $N$ become greater than $1$. Yet in this section we would like to show that there are chances to prove some non-trivial results. 

The three main ingredients in Assumption \ref{princ} are:
\begin{itemize}
\item the vector $c\in \R^s$ in the cost functional;
\item the matrix $Q\in\M^{N\times s}$ occurring in the state equation;
\item the convexification $\Lambda$ of the set of moments $L$. 
\end{itemize}

For $(c, Q)$ given, consider the set
$
\cN(c, Q)$ as it was defined in (\ref{Def_NcQ}).
Let $\Psi$ be as in (\ref{Def_function_Psi}) and such that $\nabla \Psi(m)$ is a rank $s-n$ matrix and $L$ can be seen as the embedded (parametrized) manifold of $\R^s$ in the manifold defined implicitly by $\Psi=0$. This means that $\Psi(\phi(u))=0$ for all $u\in K$. 

Consider also the set of vectors $
\cN(K, \phi) $ described in (\ref{Def_NKphi}), that is,
the set of ``ascent" directions for $\Psi$ at points of $L$.

We are now in conditions to prove Theorem~\ref{main}.

\begin{proof}
The proof is rather straightforward. Firstly, note that due to the convexity assumption on $\Psi$, and the fact that $L\subset\{\Psi=0\}$, we have $\Lambda\subset\{\Psi\le0\}$. 

Suppose that $m_0\in L$ and $m_1\in\Lambda$, so that 
$$\Psi(m_0)=0,\ \Psi(m_1)\le 0, \ cm_1\le cm_0,\text{ and }Qm_1=Qm_0\ (=\xi).$$ Then it is obvious that $m=m_1-m_0\in \cN(c, Q)$. Because of our assumption, $m\in \cN(K, \phi)$. We have two possibilities:
\begin{enumerate}
\item $\nabla \Psi(m_0)m=0$. Because of the convexity of each component of $\Psi$, we have
$$
\Psi(m_1)-\Psi(m_0)-\nabla \Psi(m_0)m\ge0.
$$
But then 
$$
0= \Psi(m_0)\le \Psi(m_1)\le0,
$$
so that $m_1\in L$. Because of the strict convexity of each component of $\Psi$, this means that $m_1=m_0$, and Assumption \ref{princ} holds. 
\item $\nabla \psi_i(m_0)m>0$ for some $i$. Once again we have
$$
\psi_i(m_1)-\psi_i(m_0)-\nabla \psi_i(m_0)m\ge0.
$$
But this is impossible because $\psi_i(m_1)>0$ cannot happen for a vector in $\Lambda$.

\end{enumerate}
\end{proof}

\begin{remark}\label{Remark_Q_Q0}
 Notice that if in the original problem $(P_1)$ we would have considered the dynamics given by $$Q(x)\phi(u)+Q_0(x)$$ instead of just $Q(x)$, Assumption~\ref{princ} and Theorem~\ref{main} could be written exactly in the same way.
\end{remark}

Though Theorem~\ref{main} can be applied to more general cases, we will focus on a particular situation motivated by the control of underwater vehicles (\cite{Bm}). We will briefly describe the structure of the state equation. Indeed, it is just
$$
x'(t)=Q_1(x)\phi(u)+Q_0(x)
$$
where the state $x\in\R^{12}$ incorporates the position and orientation in body and world coordinates, and the control $u\in\R^{10}$ accounts for guidance and propulsion. Under suitable simplifying assumptions (\cite{Bm}, \cite{Fs}, \cite{HL}), the components of the control vector $u$ only occur as either linear or pure squares, in such a way that $\phi(u)=(u, u^2)\in\R^{20}$, and $u^2=(u_i^2)_i$, componentwise. $Q_1$ and $Q_0$ are matrices which may have essentially any kind of dependence on the state $x$. 

To cover this sort of situations just described, we will concentrate on the optimal control problem $(P)$ already stated in (\ref{vectorial1})-(\ref{vectorial3}),  and set $D$, $E$ and $U$ as in (\ref{Def_Q})-(\ref{Def_DEU}).

We can now prove Theorem~\ref{sec}.
\begin{proof}

Notice that accordingly to (\ref{Def_psi_n}), as $s=2n$, we have, for $m\in \R^s$,
$$
\psi_{i}(m)=m_{i}^2-m_{n+i},\quad i=1, 2, \dots, n,
$$

which are certainly smooth and (strictly) convex. Moreover, 
$$
\nabla \Psi(m)=\begin{pmatrix} 2\tilde m & -\hbox{id}\end{pmatrix}
$$
where
$$
\tilde m=2\sum_im_ie_i\otimes e_i,
$$
and $e_i$ is the canonical basis of $\R^n$. 

Suppose we have, for a vector $v\in\R^{2n}$, $v=(v_1, v_2)$, that
$$
Qv=0,\quad cv\le0. 
$$
A more explicit way of writing this is
$$
 Q_1v_1+Q_2v_2=0,\quad c_1 v_1+ c_2v_2\le 0.
$$
So
$$
 v_1=Dv_2,\quad Ev_2\le0.
$$
We have to check that such a vector $v$ is not a direction of descent for every function $\psi_j$, or it is an ascent direction for at least one of them. Note that
$$
\nabla \Psi(m)v=Uv_2,\quad Ev_2\le0.
$$
It is an elementary Linear Algebra exercise to check that if $U^{-T}E<0$, then condition (\ref{Condition_inTeoMain}) is verified so that Theorem~\ref{main} can be applied.
\end{proof}

Corollary~\ref{ter} is a specific example of the kind of existence result that can be obtained through this approach. Its proof amounts to going carefully through the arithmetic while checking that matrix $U$ and vector $E$ defined from such given class of $(c(.),Q(.))$ verify the assumptions of Theorem~\ref{sec}.

By using the same ideas, more general situations can be treated, for example the number of controls could be greater than the components of the state. This is in fact the situation in the model that has served as an inspiration for us. We will pursue a closer analysis of such a particular situation, even stressing the more practical issues, in a forthcoming work.

\end{document}